\documentclass[12pt]{article}
\usepackage{amsmath}
\usepackage{amsthm}
\usepackage{amssymb}

\newtheorem{lemma}{Lemma}
\newtheorem{definition}{Definition}

\newtheorem{conjecture}{Conjecture}

\begin{document}

\begin{center}\Large

\textbf{Non-transversal Vectors of Some Finite Geometries}

\bigskip\large

Ivica Martinjak\\
University of Zagreb, Faculty of Science\\
Bijeni\v cka cesta 32, HR-10000 Zagreb, Croatia


\end{center}


\begin{abstract}
By means of associated structural invariants, we efficiently construct four biplanes of order 9 - except the one with the smallest automorphism group, that is found by Janko and Trung. The notion of non-transversal vector is introduced since we observed related properties that provide significantly more efficient constructions. There is a dichotomy in the structure of biplanes of order 7 and 9 with respect to the incidence matrix symmetry.

\end{abstract}

\section{Preliminaries}

We let $I_n$, $J_{m,n}$ and $0_{m,n}$ denote identity matrix, unit matrix and zero matrix, respectively, where $m$ and $n$ are dimensions of a matrix. Let $K_{m,n}$ denote a $(m +1) \times n$ matrix having all zeros in the first $m-n$ rows, than all 1s in the following row and identity matrix on the rest of entries, 
$$
K_{m,n}:=\begin{pmatrix}
0_{m-n,n}
\cr J_{1,n}
\cr I_n
\end{pmatrix}.
$$
We define $c=(c_1,\ldots,c_i,\ldots,c_n)$ as a vector with $c_i=1$ while all other entries are 0. The matrix corresponding to the permutation 
$$
\pi_1=
\begin{pmatrix}
c_1 & c_2&\ldots&c_n
\end{pmatrix}
$$
is denoted by $C_n^i$ and called {\it cyclic matrix}. The matrix defined by the permutation 
$$
\pi_2=(c_1c_{n})(c_2c_{n-1})\ldots
$$
is denoted by $C_n^{-i}$ and called {\it anticyclic matrix}. We also write $C_n^{-}$ when $i=n$. A matrix $T_{n}^i$ is consisted of principal submatrices $I_i$ and $C_{n-i}^2$,
$$
T_{n}^i:=\begin{pmatrix}
I_i &0_{i,n-i}
\cr 0_{n-i,i}& C_{n-i}^2
\end{pmatrix}.
$$
The trace of a matrix $M$ is denoted by $tr(M)$.

The scalar product $f(a,b)=a \cdot b$ of vectors $a=(a_1,a_2,\ldots,a_n)$ and $b=(b_1,b_2,\ldots,b_n)$ is defined as $f(a,b):=a_1b_1+a_2b_2+\ldots+a_nb_n$. Let $M$ be a matrix of order $m$ and $g$ a function defined as $g(M,a)=M\cdot a$. Then the {\it level set} ${\cal M}^{n,c}$ defined as
$$
g^{-1}(cJ_{n1}):=\{b : g(M,b)=c J_{n1}\}
$$
we shall call $c$-{\it space of matrix} $M$. We assume $c=2$ when this value is omitted. For a given vector $a$ that is an element of some subset of ${\cal M}^{n,c}$, the level set 
$$
f_{a,j}^{-1}(c):=\{ a \in {\cal M}_i^{n,c} : f(a,b)=c, \enspace b \in {\cal M}_j^{n,c}, \enspace  {\cal M}_i^{n,c} \cap {\cal M}_j^{n,c}=\emptyset\}
$$
is called $j$-th level set of $a$ for the value $c$.

\section{Introduction}

\begin{definition}\label{Def1}
A biplane ${\cal B}=({\cal P},{\cal L})$ is a symmetric incidence structure consisting of the set of points ${\cal P}=\{p_1,p_2,\ldots,p_v\}$ and the set of lines ${\cal L}=\{L_1,L_2,\ldots,L_v\}$ having $k$ points on every one of $v$ lines, where any two lines intersect in two points. 
\end{definition}

This means that a biplane is uniquely determined by parameters $2$-$(v,k,2)$. According to the basic properties of an {\it uniform} and {\it balanced} incidence structure, these parameters are related as 
\begin{eqnarray}
v=1+\binom{k}{2}.
\end{eqnarray} 
The order $n$ of a biplane ${\cal B}$ is defined as $n:=k-2$. 

\begin{definition}\label{Def2}
Let ${\cal B}$  be a biplane. A matrix $M=[m_{ij}]$ with dimensions $v \times v$ defined by 
$$
m_{ij}:=\begin{cases} 1 &\mbox{if } p_i \in L_j \\
0 & \mbox{if } p_i \notin L_j. \end{cases}
$$
is called incidence matrix of a biplane ${\cal B}$.
\end{definition}

A converse statement holds as well. Let $M$ be a binary matrix with dimensions $v \times v$. Then $M$ is an incidence matrix of a biplane ${\cal B}$ with parameters $2$-$(v,k,2)$ if and only if it holds
\begin{eqnarray} \label{incmatcond1}
J_{1v} \cdot M&=&kJ_{1v}\\ 
M \cdot M^{T}&=&2 J_v+(k-2)I_v \label{incmatcond2}.
\end{eqnarray}
Thus, such matrix has the sum of every column equal to $k$, the sum of every row equals also to $k$ and the property that scalar product of every two rows is equal to 2. Throughout this work a biplane is represented by its {\it incidence matrix}.

For a biplane ${\cal B}$ we define a {\it dual} biplane ${\cal B}^{T}=({\cal P}^{T},{\cal L}^{T})$ such that ${\cal P}^{T} = {\cal L}$ and ${\cal L}^{T} = {\cal P}$. When a biplane ${\cal B}$ is represented by an incidence matrix $M$, the dual ${\cal B}^{T}$ of ${\cal B}$ is represented by a transpose matrix $M^{T}$.

Biplanes ${\cal B}$ and ${\cal B}'$ are {\it isomorphic} if there exists an isomorphism from ${\cal B}$ to ${\cal B}'$. An isomorphism from ${\cal B}$ to ${\cal B}$ is called {\it automorphism}. The set of all automorphisms of a biplane ${\cal B}$ is called the {\it full automorphism group}, denoted $Aut({\cal B})$. 

We let ${\cal B}_{na}$, ${\cal B}_{nb}$ denote biplanes of order $n$ in respect to ascending order of $|Aut(\cal B)|$. In particular, we let ${\cal B}_{n}$ denote a unique biplane of order $n$, in short. There are $17$ known biplanes, with $7$ different orders \cite{MaRo}. These are ${\cal B}_{1}$, ${\cal B}_{2}$, ${\cal B}_{3}$, ${\cal B}_{4a}$, ${\cal B}_{4b}$, ${\cal B}_{4c}$, ${\cal B}_{7a}$, \ldots, ${\cal B}_{7d}$, ${\cal B}_{9a}$, \ldots, ${\cal B}_{9e}$, ${\cal B}_{11a}$, ${\cal B}_{11b}$. It is known that there are infinitely many projective planes whereas for biplanes, triplanes and other finite geometries this question is open. The conjecture is that there are finitely many biplanes \cite{Cam}. More about biplanes one can find in \cite{HuPi} and \cite{Sal1} while the broader context is provided in \cite{Stin} and \cite{McKay}. It is worth mentioning that biplanes are very regular structures with many relations with other combinatorial and algebraic structures including graphs, hypergraphs, difference sets and association schemes \cite{BlBr, Mart}.

The structure with the largest known parameters is one of order $11$, found by M. Aschbacher \cite{Asc}. A complete classification of biplanes having 56 points is recently done by P. Kaski and P. R. J. \"{O}sterg\r and \cite{KO} while the last new biplane is ${\cal B}_{9a}$ having $|Aut({\cal B}_{9a})|=24$ that is found by Z. Janko and T. Trung \cite{JT}. This work is the result of our aim at finding efficient constructions of biplanes of higher orders.

It is known that a few first rows within a finite geometry can be determined without loosing generality, up to isomorphism. In case of biplanes, it is possible to determine first $k$ rows is that sense. Let ${\cal B}$ be a biplane of order $k-2$ and let $M=[m_{ij}]$ be its incidence matrix. Then the first $k$ rows and the first $k$ columns of $M$ are uniquely determined, up to isomorphism,
\begin{eqnarray} \label{CanMat}
M&=&
\begin{pmatrix}
J_{k,1} K_{k,k-1} & K_{k,k-2} & \ldots & K_{k,1} \cr
K_{k,k-2}^T& D^{1,1} & \ldots & D^{1,k-2} \cr
\ldots&  \ldots& \ldots&\ldots \cr
 K_{k,1}^T&D^{k-2,1} & \ldots & D^{k-2,k-2}
\end{pmatrix}.
\end{eqnarray}
The rest of elements let be separated into submatrices $D^{i,j}$, $i,j=1,\ldots,k-2$. We say that an incidence matrix $M$ (\ref{CanMat}) is in {\it canonical form}. A canonical incidence matrix of the biplane ${\cal B}_{4c}$ is depicted on Figure \ref{Fig0}  (with $\cdot$ instead of 0).


\begin{figure}[h!]
\begin{center}
\begin{scriptsize}
\begin{tabular}{|c@{\hspace{0.5em}}c@{\hspace{0.5em}}c@{\hspace{0.5em}}c@{\hspace{0.5em}}c@{\hspace{0.5em}}c@{\hspace{0.5em}}|c@{\hspace{0.5em}}c@{\hspace{0.5em}}c@{\hspace{0.5em}}c@{\hspace{0.5em}}|c@{\hspace{0.5em}}c@{\hspace{0.5em}}c@{\hspace{0.5em}}|c@{\hspace{0.5em}}c@{\hspace{0.5em}}|c@{\hspace{0.5em}}|} \hline
1 & 1 & 1 & 1 & 1 & 1 & $\cdot$ & $\cdot$ &  $\cdot$ &  $\cdot$ &  $\cdot$ & $ \cdot$ &  $\cdot$ & $\cdot$ & $\cdot$ & $\cdot$ \\
1 & 1 & $\cdot$ & $\cdot$& $\cdot$ & $\cdot$ & 1 & 1 & 1 & 1 &  $\cdot$ & $ \cdot$ &  $\cdot$ & $\cdot$ & $\cdot$ & $\cdot$ \\
1 & $\cdot$ & 1 & $\cdot$& $\cdot$ & $\cdot$ & 1 & $\cdot$ & $\cdot$ & $\cdot$ &  1 & 1 &  1 & $\cdot$ & $\cdot$ & $\cdot$ \\
1 & $\cdot$ & $\cdot$ & 1 & $\cdot$ & $\cdot$ & $\cdot$ & 1 & $\cdot$ & $\cdot$ &  1 & $\cdot$ & $\cdot$ & 1 &  1   & $\cdot$ \\
1 & $\cdot$ & $\cdot$ & $\cdot$ & 1 & $\cdot$ & $\cdot$ & $\cdot$ & 1 & $\cdot$ &  $\cdot$ & 1 & $\cdot$ & 1 &  $\cdot$   & 1 \\
1 & $\cdot$ & $\cdot$ & $\cdot$ &  $\cdot$ & 1 & $\cdot$ & $\cdot$ & $\cdot$ & 1 &  $\cdot$ & $\cdot$ & 1 &    $\cdot$ & 1  & 1\\ \hline
$\cdot$  & 1 & 1 & $\cdot$ &  $\cdot$ & $\cdot$  & 1 & $\cdot$ & $\cdot$ & $\cdot$  &  $\cdot$ & $\cdot$ & $\cdot$  & 1 & 1  & 1\\
$\cdot$  & 1 &  $\cdot$  & 1 &  $\cdot$ & $\cdot$  & $\cdot$ & 1 &  $\cdot$ & $\cdot$  &  $\cdot$ & 1 & 1 & $\cdot$ & $\cdot$   & 1\\
$\cdot$  & 1 &  $\cdot$  &  $\cdot$ & 1 &  $\cdot$  & $\cdot$ & $\cdot$ & 1 & $\cdot$  & 1 &  $\cdot$ & 1 & $\cdot$  & 1 & $\cdot$  \\
$\cdot$  & 1 &  $\cdot$  &  $\cdot$ &   $\cdot$ & 1  & $\cdot$ & $\cdot$ &  $\cdot$ & 1  & 1 &  1 & $\cdot$ &  1 &  $\cdot$& $\cdot$  \\ \hline
$\cdot$  & $\cdot$  &  1  &  1 &   $\cdot$ & $\cdot$ & $\cdot$ & $\cdot$ &  1 & 1  & 1 &  $\cdot$ & $\cdot$ &  $\cdot$ &  $\cdot$& 1  \\
$\cdot$  & $\cdot$  &  1  &  $\cdot$ &   1 & $\cdot$ & $\cdot$ & 1 &  $\cdot$ & 1  & $\cdot$ &  $\cdot$ & 1 &1&  $\cdot$ &  $\cdot$  \\
$\cdot$  & $\cdot$  &  1  &  $\cdot$  &    1 &$\cdot$ & $\cdot$ & 1 &  1  & $\cdot$  & $\cdot$ &  1 & $\cdot$  & $\cdot$ & 1  &  $\cdot$  \\ \hline
$\cdot$  & $\cdot$  &    $\cdot$  &  1  &  1 &$\cdot$ & 1 & $\cdot$  & $\cdot$ &  1   & $\cdot$ &  1 & $\cdot$  & $\cdot$ & 1  &  $\cdot$  \\
$\cdot$  & $\cdot$  &    $\cdot$  &  1  &  $\cdot$ & 1  & 1 & $\cdot$  &  1  & $\cdot$   & $\cdot$ &  $\cdot$ & 1 & 1 & $\cdot$   &  $\cdot$  \\ \hline
$\cdot$  & $\cdot$  &    $\cdot$  & $\cdot$  &  1 & 1  & 1 & 1 &  $\cdot$  & $\cdot$   & 1 &  $\cdot$ & $\cdot$ & $\cdot$ & $\cdot$   &  1 \\ \hline
\end{tabular}
\end{scriptsize}
\caption{Scheme of a canonical incidence matrix of the biplane ${\cal B}_{4c}$.}\label{Fig0}
\end{center}
\end{figure}


\section{Constructions with Associated Structural Invariants}

Let $M_{[k]}$ denotes canonical rows of $M$, i.e. the first $k$ rows of $M$.

Now, we define 2-space of biplane's canonical incidence matrix $M$, as the set of all (0,1)-vectors having scalar product equal to 2 with every row of $M_{[k]}$.
\begin{definition}
Let ${\cal B}$ be a biplane of order $k$ represented by a canonical incidence matrix $M$. Then the set  
$
{\cal M}^v:=\{b : g(M_{[k]},b)=2J_{k1}\}
$
is called 2-space of the biplane ${\cal B}$.  
\end{definition}

In this work we perform biplane constructions based on the set ${\cal \grave{M}}^{v} \subset {\cal M}^v$ that we define as
\begin{eqnarray}
 {\cal \grave{M}}^{v}:=   \{b \in {\cal M}_i^v : b[i]=1, \enspace  i=k+1,\ldots,v \}.
\end{eqnarray}
The following section shows that such constructions are efficient and that property of having trace equal to the number of points exists within biplanes with large parameters.
It can be shown that for a biplane of order $k-2$ the subsets ${\cal \grave{M}}_i^{v} \subset {\cal \grave{M}}^{v}$ have the same cardinality,
$$
|{\cal \grave{M}}_i^{v}|=q, \enspace  i=k+1,\ldots,v, \enspace q \in \mathbb N.
$$
Consequently, the upper bound for $|{\cal \grave{M}}_i^{v}|$ is 
\begin{eqnarray*}
 |{\cal \grave{M}}_i^{v}| \leq   \binom{k-2}{2}\binom{v-3k+5}{k-5}.
\end{eqnarray*}
 More precisely, our experiments shows that for the known orders $1$, $2$, $3$, $4$, $7$, $9$, $11$ of biplanes, $q$ is equal to $1$, $0$, $0$, $1$, $70$, $3507$, $286884$, respectively.

When all entries on the main diagonal are equal to $1$, then submatrices $D^{i,j}$ contain the same pattern of zeros as matrices $K_{mn}$ of 1s. This fact is expressed by Lemma \ref{Cor2}. Dual statements of relations (\ref{lemma1aid}) and (\ref{lemma1bid}) hold as well.

\begin{lemma} \label{Cor2}
Let ${\cal B}$ be a biplane of order $k-2$ and $M$ its canonical incidence matrix having trace $tr(M)=1+\binom{k}{2}$. Then we have
\begin{eqnarray}
D^{i,j}[j-i,l]=0, \enspace \forall l, \enspace i < j,\label{lemma1aid}\\ 
D^{i,j}[j-i+l,l]=0, \enspace \forall l, \enspace i < j, \label{lemma1bid}
\end{eqnarray}
where $i,j=1,\ldots,k-2$.
\end{lemma}

\begin{proof}
According to the definition of $M$ (\ref{CanMat}), submatrices $K_{k,k-j}$ and $D^{i,j-1}$ start at the same column $p$ of $M$ and have the same number of columns. The number $p$ takes values 
$$k+1, 2k-1, 3k-4, 4k-8,\ldots$$ for $K_{k,2},K_{k,3},K_{k,4},K_{k,5}\ldots$, respectively. Apparently, constant terms in this sequence arise from the recurrence relation defined as  
\begin{eqnarray*}
a(n)&=&a(n-1)+(n+1),\\
a(0)&=&k+1.
\end{eqnarray*}
Furthermore, we have
$$
p=(j-1)k - \frac{j(j+1)}{2}-2.
$$
Similarly, it follows from (\ref{CanMat}) that the entry $D^{i,j-1}[j-i,1]$ is on the $r$-th row of $M$,
$$
r=ik-\frac{i(i+1)}{2}+j-i.
$$
Now we demonstrate that $j$-th and $r$-th row have two 1s in common in the positions $1,2, \ldots, p-1$. Obviously, both of these rows have 1s in the $j$-th position. In general, $j$-th row, $j \leq k$ has 1s in the positions $$1, j, j+k-2, j+2k-5, \ldots,(j-2)k-\frac{j(j-3)}{2}+1.$$
On the other hand, the $r$-th row has 1s in the $r$-th position (the main diagonal). The fact that this position is present in the above sequence completes the first statement of the proof. Similar reasoning proves the second statement. 
\end{proof}

Lemma \ref{Cor2} gives an immediate construction of a biplane of order 4 admitting an incidence matrix $M$ with all entries on the main diagonal equal to $1$. In the first step we form the first $k$ rows and the first $k$ columns. Once assuming that $tr(M)=16$, Lemma \ref{Cor2} defines positions of all zeros in submatrices $D^{i,j}$. Finally, when we put 1s on the remaining $4$ places in every row (column), $M$ meets conditions (\ref{incmatcond1})-(\ref{incmatcond2}). We obtained the biplane ${\cal B}_{4c}$ of order 4 with $|Aut({\cal B}_{4c})|=11520$. Figure \ref{Fig0} presents this biplane in a matrix form.

Biplanes of order $9$ were the subject of intensive research, mostly by algebraic approach \cite{EsCe,EsCe4,HLW, Marg}. Recent classification done in \cite{KO} confirmed that five already known biplanes are the only biplanes of this order. For the purpose to efficiently construct biplane of this order we associate block matrices as a {\it structural invariant} of a biplane. 

Let $P_n$ be an adjacency matrix of a {\it path} with $n$ vertices having starting and ending vertex having {\it loop}. Then $L_n:=P_n C_n^{-}$.  In other words,
$$
\qquad
L_{n}=\begin{pmatrix}
1 & 1 & &  &
\cr   1& & 1 & & 
\cr   &\ddots &  &\ddots & 
\cr   &   &1 & &1
\cr   &   && 1&1
\end{pmatrix}. 
$$

With this notation, for the submatrix $D^{1,2}$ of the only canonical incidence matrix of ${\cal B}_{4c}$ it holds true $D^{1,2}=A$,
$$
A:= \left( \begin{array}{c}
0_{1,3}\\
L_3  \end{array}  \right).
$$
This give us a hint for further invariants, within biplanes of higher orders. Using our own algorithm we performed constructions for biplanes of order $7$ and $9$. Recall that biplanes of order 7 were classified in work \cite{Sal2}. In both cases we set
\begin{eqnarray}
D^{1,1}=I_{k-2}. 
\end{eqnarray}
In case of biplanes with $37$ points the second associated invariant is $D^{1,2}=B$,
$$
B:= \left( \begin{array}{cc}
0_{1,6}\\
0_3 & L_3\\
L_3 & 0_3 \end{array}  \right).
$$
Similarly, in case of biplanes of order 9 we set $D^{1,2}=C$,
$$
C:= \left( \begin{array}{cc}
0_{1,8}\\
0_4 & L_4 C_4^{-}\\
L_4 & 0_4 \end{array}  \right).
$$

\begin{figure}[h!]
\begin{center}
\begin{scriptsize}
\begin{tabular}{|c@{\hspace{0.5em}}c@{\hspace{0.5em}}c@{\hspace{0.5em}}c@{\hspace{0.5em}}c@{\hspace{0.5em}}c@{\hspace{0.5em}}c@{\hspace{0.5em}}|c@{\hspace{0.5em}}c@{\hspace{0.5em}}c@{\hspace{0.5em}}c@{\hspace{0.5em}}c@{\hspace{0.5em}}c@{\hspace{0.5em}}|} \hline
1 & $\cdot$ & $\cdot$ & $\cdot$ & $\cdot$ & $\cdot$ & $\cdot$ & $\cdot$ &  $\cdot$ &  $\cdot$ &  $\cdot$ & $ \cdot$ &  $\cdot$  \\
$\cdot$ & 1 & $\cdot$ & $\cdot$& $\cdot$ & $\cdot$ & $\cdot$ & $\cdot$ & 1& $\cdot$ &  1 & $ \cdot$ &  $\cdot$  \\
$\cdot$ & $\cdot$ & 1 & $\cdot$& $\cdot$ & $\cdot$ & $\cdot$ & $\cdot$ & $\cdot$ & 1&  $\cdot$ & 1 &  $\cdot$ \\
$\cdot$ & $\cdot$ & $\cdot$ & 1 & $\cdot$ & $\cdot$ & $\cdot$ & 1 & $\cdot$ & $\cdot$ &  $\cdot$ & $\cdot$ & 1  \\
$\cdot$ & $\cdot$ & $\cdot$ & $\cdot$ & 1 & $\cdot$ & $\cdot$ & 1& $\cdot$ & $\cdot$ &  $\cdot$ & 1& $\cdot$  \\
$\cdot$ & $\cdot$ & $\cdot$ & $\cdot$ &  $\cdot$ & 1 & $\cdot$ & $\cdot$ & 1& $\cdot$ &  $\cdot$ & $\cdot$ & 1    \\ 
$\cdot$  & $\cdot$ & $\cdot$ & $\cdot$ &  $\cdot$ & $\cdot$  & 1 & $\cdot$ & $\cdot$ & 1  &  1 & $\cdot$ & $\cdot$   \\ \hline
\end{tabular}
\end{scriptsize}
\caption{Matrices $D^{1,1}$ and $D^{1,2}$ that were associated structural invariants for duals of biplanes of order 7.}\label{FigStrucIvariantForB7}
\end{center}
\end{figure}
\begin{figure}[h!]
\begin{center}
\begin{scriptsize}
\begin{tabular}{|c@{\hspace{0.5em}}c@{\hspace{0.5em}}c@{\hspace{0.5em}}c@{\hspace{0.5em}}c@{\hspace{0.5em}}c@{\hspace{0.5em}}c@{\hspace{0.5em}}c@{\hspace{0.5em}}c@{\hspace{0.5em}}|c@{\hspace{0.5em}}c@{\hspace{0.5em}}c@{\hspace{0.5em}}c@{\hspace{0.5em}}c@{\hspace{0.5em}}c@{\hspace{0.5em}}c@{\hspace{0.5em}}c@{\hspace{0.5em}}|} \hline
1 & $\cdot$ & $\cdot$ & $\cdot$ & $\cdot$ & $\cdot$ & $\cdot$ & $\cdot$ &  $\cdot$ &  $\cdot$ &  $\cdot$ & $ \cdot$ &  $\cdot$ & $\cdot$ & $\cdot$ & $\cdot$ & $\cdot$\\
$\cdot$ & 1 & $\cdot$ & $\cdot$& $\cdot$ & $\cdot$ & $\cdot$ & $\cdot$ & $\cdot$ & $\cdot$ &  1 & $ \cdot$ &  $\cdot$ & $\cdot$ & $\cdot$ & $\cdot$ & 1\\
$\cdot$ & $\cdot$ & 1 & $\cdot$& $\cdot$ & $\cdot$ & $\cdot$ & $\cdot$ & $\cdot$ & $\cdot$ &  $\cdot$ & 1 &  $\cdot$ & 1& $\cdot$ & $\cdot$ & $\cdot$\\
$\cdot$ & $\cdot$ & $\cdot$ & 1 & $\cdot$ & $\cdot$ & $\cdot$ & $\cdot$ & $\cdot$ & $\cdot$ &  $\cdot$ & $\cdot$ & 1 & $\cdot$ &  1  & $\cdot$ & $\cdot$\\
$\cdot$ & $\cdot$ & $\cdot$ & $\cdot$ & 1 & $\cdot$ & $\cdot$ & $\cdot$ & $\cdot$ & 1 &  $\cdot$ & $\cdot$ & $\cdot$ & $\cdot$ &  $\cdot$   & 1 & $\cdot$\\
$\cdot$ & $\cdot$ & $\cdot$ & $\cdot$ &  $\cdot$ & 1 & $\cdot$ & $\cdot$ & $\cdot$ & 1 &  $\cdot$ & $\cdot$ & $\cdot$ &    $\cdot$ & 1  & $\cdot$& $\cdot$\\ 
$\cdot$  & $\cdot$ & $\cdot$ & $\cdot$ &  $\cdot$ & $\cdot$  & 1 & $\cdot$ & $\cdot$ & $\cdot$  &  1 & $\cdot$ & $\cdot$  & $\cdot$ & $\cdot$  & 1 & $\cdot$\\
$\cdot$  & $\cdot$ &  $\cdot$  & $\cdot$ &  $\cdot$ & $\cdot$  & $\cdot$ & 1 &  $\cdot$ & $\cdot$  &  $\cdot$ & 1 & $\cdot$ & $\cdot$ & $\cdot$   & $\cdot$& 1\\
$\cdot$  & $\cdot$ &  $\cdot$  &  $\cdot$ & $\cdot$ &  $\cdot$  & $\cdot$ & $\cdot$ & 1 & $\cdot$  & $\cdot$ &  $\cdot$ & 1 & 1  & $\cdot$ & $\cdot$  &$\cdot$ \\ \hline
\end{tabular}
\end{scriptsize}
\caption{Matrices $D^{1,1}$ and $D^{1,2}$ that were associated structural invariants for ${\cal B}_{9c}$.}\label{FigStrucIvariantForB9c}
\end{center}
\end{figure}

It has proved that this approach is fruitful. Obtained results are presented in Table \ref{Tab1}. In case of biplanes of order 7, one of duals is constructed while in case of biplanes of order $9$ we find three structures, having automorphism group orders $64$, $288$ and $80640$. 

\begin{table}[h]
$$
\begin{array}{cccl}
n & D^{1,1} & D^{1,2} & |Aut({\cal B})|
\cr 4 & I_4 & A & 11520
\cr 7 & I_7 & B & 1512
\cr 9 & I_9 & C & 64, 288, 80640
\end{array}
$$
\caption{Constructions with associated structural invariants.}\label{Tab1}
\end{table}

It is worth mentioning that we also experiment with some other associated structural invariants. In particular, when associate a structural invariant for $D^{1,2}$ that is depicted in Figure \ref{FigStrucIvariantForB7} (with $\cdot$ instead of 0), in case of biplanes of order 7, we get both duals. 

When associate invariant for $D^{1,2}$ presented at Figure \ref{FigStrucIvariantForB9c} (with $\cdot$ instead of 0) in case of biplanes of order 8, we get the biplane ${\cal B}_{9c}$, having the automorphism group order $\vert Aut({\cal B}_{9c}) \vert=144$.

Figure \ref{Fig1} shows a constructed incidence matrix of the biplane ${\cal B}_{9e}$, without canonical rows and columns (1s are stated, all other entries are 0).

\begin{figure}[h!]
\begin{center}
\begin{tiny}
\begin{tabular}{|c@{\hspace{0.4em}}c@{\hspace{0.4em}}c@{\hspace{0.4em}}c@{\hspace{0.4em}}c@{\hspace{0.4em}}c@{\hspace{0.4em}}c@{\hspace{0.4em}}c@{\hspace{0.4em}}c@{\hspace{0.4em}}|c@{\hspace{0.4em}}c@{\hspace{0.4em}}c@{\hspace{0.4em}}c@{\hspace{0.4em}}c@{\hspace{0.4em}}c@{\hspace{0.4em}}c@{\hspace{0.4em}}c@{\hspace{0.4em}}|c@{\hspace{0.4em}}c@{\hspace{0.4em}}c@{\hspace{0.4em}}c@{\hspace{0.4em}}c@{\hspace{0.4em}}c@{\hspace{0.4em}}c@{\hspace{0.4em}}|c@{\hspace{0.4em}}c@{\hspace{0.4em}}c@{\hspace{0.4em}}c@{\hspace{0.4em}}c@{\hspace{0.4em}}c@{\hspace{0.4em}}|c@{\hspace{0.4em}}c@{\hspace{0.4em}}c@{\hspace{0.4em}}c@{\hspace{0.4em}}c@{\hspace{0.4em}}|c@{\hspace{0.4em}}c@{\hspace{0.4em}}c@{\hspace{0.4em}}c@{\hspace{0.4em}}|c@{\hspace{0.4em}}c@{\hspace{0.4em}}c@{\hspace{0.4em}}|c@{\hspace{0.4em}}c@{\hspace{0.4em}}|c@{\hspace{0.4em}}|} \hline
1&$\cdot$ &$\cdot$ &$\cdot$ &$\cdot$&$\cdot$ &$\cdot$ &$\cdot$ &$\cdot$ &$\cdot$ &$\cdot$ &$\cdot$ &$\cdot$ &$\cdot$	 &$\cdot$ &$\cdot$ &$\cdot$ &1&	1&$\cdot$ &$\cdot$ &$\cdot$ &$\cdot$ &$\cdot$ &	$\cdot$ &	1&$\cdot$ &$\cdot$ &$\cdot$ &$\cdot$ &	1&$\cdot$ &$\cdot$ &$\cdot$ &$\cdot$ &$\cdot$ &	$\cdot$&	$\cdot$ &$\cdot$ &	1&	1&$\cdot$&$\cdot$ &1&	1\\
$\cdot$&1& $\cdot$&$\cdot$ &$\cdot$&$\cdot$ &$\cdot$&$\cdot$&$\cdot$&$\cdot$&$\cdot$ &$\cdot$&$\cdot$ &$\cdot$ & $\cdot$&	1&	1&$\cdot$ &$\cdot$ &$\cdot$ &$\cdot$ &$\cdot$ &$\cdot$ &$\cdot$ &	1&$\cdot$ &$\cdot$ &	1&$\cdot$ &$\cdot$ &$\cdot$ &	1&$\cdot$ &$\cdot$ &$\cdot$ &$\cdot$ &$\cdot$ &	1&	1&	1&$\cdot$ &$\cdot$ &$\cdot$ &$\cdot$ &$\cdot$ \\
$\cdot$ &$\cdot$	 &	1&	$\cdot$ &$\cdot$	 &$\cdot$	 &$\cdot$	 &$\cdot$	 &$\cdot$	 &$\cdot$	 &$\cdot$	 &$\cdot$	 &$\cdot$	 &$\cdot$	 &	1&$\cdot$	 &	1&$\cdot$ &$\cdot$ &	1&$\cdot$ &$\cdot$ &	1&$\cdot$ &$\cdot$ &$\cdot$ &$\cdot$ &$\cdot$ &$\cdot$ &$\cdot$ &$\cdot$ &$\cdot$ &	1&$\cdot$ &	1&	1&$\cdot$ &$\cdot$ &	$\cdot$ &$\cdot$ &	1&$\cdot$ &$\cdot$ &$\cdot$ &$\cdot$ \\
$\cdot$ &$\cdot$	 &$\cdot$	 &	1&$\cdot$	 &$\cdot$	 &$\cdot$	 &$\cdot$	 &$\cdot$	 &$\cdot$	 &$\cdot$	 &$\cdot$	 &$\cdot$	 &	1&$\cdot$	 &	1&$\cdot$	 &$\cdot$ &$\cdot$ &	1&$\cdot$ &$\cdot$ &$\cdot$ &	1&$\cdot$ &$\cdot$ &	1&$\cdot$ &	1&$\cdot$ &$\cdot$ &$\cdot$ &$\cdot$ &$\cdot$ &$\cdot$ &$\cdot$ &	1&$\cdot$ &$\cdot$ &$\cdot$ &$\cdot$ &$\cdot$ &$\cdot$ &	1&$\cdot$ \\
$\cdot$ &	$\cdot$ &$\cdot$	 &$\cdot$	 &	1&$\cdot$	 &$\cdot$	 &$\cdot$	 &$\cdot$	 &$\cdot$	 &	$\cdot$ &$\cdot$	 &	$\cdot$ &	1&	1&	$\cdot$ &	$\cdot$ &$\cdot$ &$\cdot$ &$\cdot$ &	1&	1&$\cdot$ &$\cdot$ &	1&$\cdot$ &$\cdot$ &$\cdot$ &$\cdot$ &	1&$\cdot$ &$\cdot$ &$\cdot$ &	1&$\cdot$ &$\cdot$ &$\cdot$ &$\cdot$ &$\cdot$ &$\cdot$ &$\cdot$ &$\cdot$ &$\cdot$ &$\cdot$ &	1\\
$\cdot$ &	$\cdot$ &$\cdot$	 &$\cdot$	 &$\cdot$	 &	1&$\cdot$	 &$\cdot$	 &$\cdot$	 &	1&	1&$\cdot$	 &$\cdot$	 &$\cdot$	 &	$\cdot$ &$\cdot$	 &$\cdot$	 &$\cdot$ &$\cdot$ &$\cdot$ &$\cdot$ &$\cdot$ &$\cdot$ &	1&$\cdot$ &$\cdot$ &$\cdot$ &$\cdot$ &$\cdot$ &	1&	1&$\cdot$ &	1&$\cdot$ &$\cdot$ &$\cdot$ &$\cdot$ &	1&$\cdot$ &$\cdot$ &$\cdot$ &$\cdot$ &	1&$\cdot$ &$\cdot$ \\
$\cdot$ &$\cdot$	 &$\cdot$	 &$\cdot$	 &$\cdot$	 &$\cdot$	 &	1&$\cdot$	 &$\cdot$	 &	1&$\cdot$	 &	1&$\cdot$	 &$\cdot$	 &$\cdot$	 &$\cdot$	 &$\cdot$	 &$\cdot$ &$\cdot$ &$\cdot$ &$\cdot$ &$\cdot$ &	1&$\cdot$ &$\cdot$ &	1&	1&$\cdot$ &$\cdot$ &$\cdot$ &$\cdot$ &$\cdot$ &$\cdot$ &	1&$\cdot$ &$\cdot$ &$\cdot$ &$\cdot$ &	1&$\cdot$ &$\cdot$ &	1&$\cdot$ &$\cdot$ &$\cdot$ \\
 $\cdot$&$\cdot$	 &$\cdot$	 &$\cdot$	 &$\cdot$	 &$\cdot$	 &$\cdot$	 &	1&$\cdot$	 &$\cdot$	 &	1&$\cdot$	 &	1&$\cdot$	 &	$\cdot$ &$\cdot$	 &$\cdot$	 &$\cdot$ &	1&$\cdot$ &	1&$\cdot$ &$\cdot$ &$\cdot$ &$\cdot$ &$\cdot$ &$\cdot$ &	1&$\cdot$ &$\cdot$ &$\cdot$ &$\cdot$ &$\cdot$ &$\cdot$ &	1&$\cdot$ &	1&$\cdot$ &$\cdot$ &$\cdot$ &$\cdot$ &	1&$\cdot$ &$\cdot$ &$\cdot$ \\
$\cdot$ &$\cdot$	 &$\cdot$	 &$\cdot$	 &$\cdot$	 &$\cdot$	 &$\cdot$	 &$\cdot$	 &	1&$\cdot$	 &$\cdot$	 &	1&	1&	$\cdot$ &$\cdot$	 &$\cdot$	 &$\cdot$	 &	1&$\cdot$ &$\cdot$ &$\cdot$ &	1&$\cdot$ &$\cdot$ &$\cdot$ &$\cdot$ &$\cdot$ &$\cdot$ &	1&$\cdot$ &$\cdot$ &	1&$\cdot$ &$\cdot$ &$\cdot$ &	1&$\cdot$ &$\cdot$ &$\cdot$ &$\cdot$ &$\cdot$ &$\cdot$ &	1&$\cdot$ &$\cdot$ \\ \hline
$\cdot$ &$\cdot$	 &$\cdot$	 &$\cdot$	 &$\cdot$	 &	1&	1&	$\cdot$ &	$\cdot$ &	1&$\cdot$	 &$\cdot$	 &$\cdot$	 &$\cdot$	 &$\cdot$	 &$\cdot$	 &$\cdot$	 &$\cdot$ &$\cdot$ &$\cdot$ &$\cdot$ &$\cdot$ &$\cdot$ &$\cdot$ &	1&$\cdot$ &$\cdot$ &$\cdot$ &	1&$\cdot$ &$\cdot$ &$\cdot$ &$\cdot$ &$\cdot$ &	1&	1&	1&$\cdot$ &$\cdot$ &$\cdot$ &$\cdot$ &$\cdot$ &$\cdot$ &$\cdot$ &	1\\
$\cdot$ &	$\cdot$ &$\cdot$	 &$\cdot$	 &$\cdot$	 &	1&$\cdot$	 &	1&$\cdot$	 &$\cdot$	 &	1&$\cdot$	 &$\cdot$	 &$\cdot$	 &$\cdot$	 &$\cdot$	 &$\cdot$	 &$\cdot$ &$\cdot$ &	1&$\cdot$ &	1&$\cdot$ &$\cdot$ &$\cdot$ &$\cdot$ &$\cdot$ &$\cdot$ &$\cdot$ &$\cdot$ &$\cdot$ &	1&$\cdot$ &	1&$\cdot$ &$\cdot$ &$\cdot$ &$\cdot$ &	1&$\cdot$ &$\cdot$ &$\cdot$ &$\cdot$ &	1&$\cdot$ \\
$\cdot$ &$\cdot$	 &$\cdot$	 &$\cdot$	 &$\cdot$	 &$\cdot$	 &	1&$\cdot$	 &	1&$\cdot$	 &$\cdot$	 &	1&	$\cdot$ &$\cdot$	 &$\cdot$	 &$\cdot$	 &$\cdot$	 &$\cdot$ &$\cdot$ &	1&	1&$\cdot$ &$\cdot$ &$\cdot$ &$\cdot$ &$\cdot$ &$\cdot$ &	1&$\cdot$ &	1&$\cdot$ &$\cdot$ &$\cdot$ &$\cdot$ &$\cdot$ &$\cdot$ &$\cdot$ &	1&$\cdot$ &$\cdot$ &	1&$\cdot$ &$\cdot$ &$\cdot$ &$\cdot$ \\
$\cdot$ &	$\cdot$ &	$\cdot$ &$\cdot$	 &$\cdot$	 &$\cdot$	 &$\cdot$	 &	1&	1&$\cdot$	 &$\cdot$	 &$\cdot$	 &	1&$\cdot$	 &	$\cdot$ &	$\cdot$ &$\cdot$	 &$\cdot$ &$\cdot$ &$\cdot$ &$\cdot$ &$\cdot$ &	1&	1&	1&$\cdot$ &	1&$\cdot$ &$\cdot$ &$\cdot$ &$\cdot$ &$\cdot$ &	1&$\cdot$ &$\cdot$ &$\cdot$ &$\cdot$ &$\cdot$ &	 &	1&$\cdot$ &$\cdot$ &$\cdot$ &$\cdot$ &$\cdot$ \\ 
$\cdot$ &$\cdot$	 &$\cdot$	 &	1&	1&$\cdot$	 &$\cdot$	 &$\cdot$	 &$\cdot$	 &$\cdot$	 &$\cdot$	 &	$\cdot$ &$\cdot$	 &	1&$\cdot$	 &$\cdot$	 &	$\cdot$ &	1&$\cdot$ &$\cdot$ &$\cdot$ &$\cdot$ &	1&$\cdot$ &$\cdot$ &$\cdot$ &$\cdot$ &	1&$\cdot$ &$\cdot$ &$\cdot$ &$\cdot$ &$\cdot$ &$\cdot$ &	1&$\cdot$ &$\cdot$ &$\cdot$ &	1&$\cdot$ &$\cdot$ &$\cdot$ &	1&$\cdot$ &$\cdot$ \\
$\cdot$ &$\cdot$	 &	1&$\cdot$	 &	1&$\cdot$	 &$\cdot$	 &$\cdot$	 &$\cdot$	 &$\cdot$	 &$\cdot$	 &$\cdot$	 &$\cdot$	 &$\cdot$	 &	1&	$\cdot$ &$\cdot$	 &$\cdot$ &	1&$\cdot$ &$\cdot$ &$\cdot$ &$\cdot$ &	1&$\cdot$ &$\cdot$ &$\cdot$ &$\cdot$ &	1&$\cdot$ &$\cdot$ &	1&$\cdot$ &$\cdot$ &$\cdot$ &$\cdot$ &$\cdot$ &	1&$\cdot$ &$\cdot$ &$\cdot$ &	1&$\cdot$ &$\cdot$ &$\cdot$ \\
$\cdot$ &	1&$\cdot$	 &	1&$\cdot$	 &$\cdot$	 &$\cdot$	 &$\cdot$	 &$\cdot$	 &$\cdot$	 &	$\cdot$ &	$\cdot$ &$\cdot$	 &$\cdot$	 &$\cdot$	 &	1&$\cdot$	 &$\cdot$ &$\cdot$ &$\cdot$ &$\cdot$ &	1&$\cdot$ &$\cdot$ &$\cdot$ &	1&$\cdot$ &$\cdot$ &$\cdot$ &	1&$\cdot$ &$\cdot$ &	1&$\cdot$ &$\cdot$ &	1&$\cdot$ &$\cdot$ &	 &$\cdot$ &$\cdot$ &	1&$\cdot$ &$\cdot$ &$\cdot$ \\
$\cdot$ &	1&	1&$\cdot$	 &$\cdot$	 &$\cdot$	 &$\cdot$	 &$\cdot$	 &$\cdot$	 &$\cdot$	 &$\cdot$	 &$\cdot$	 &$\cdot$	 &$\cdot$	 &$\cdot$	 &$\cdot$	 &	1&$\cdot$ &$\cdot$ &$\cdot$ &	1&$\cdot$ &$\cdot$ &$\cdot$ &$\cdot$ &$\cdot$ &	1&$\cdot$&$\cdot$ &$\cdot$ &	1&$\cdot$ &$\cdot$ &	1&$\cdot$ &$\cdot$ &	1&$\cdot$ &$\cdot$ &$\cdot$ &$\cdot$ &$\cdot$ &	1&$\cdot$ &$\cdot$ \\ \hline
1&$\cdot$ &$\cdot$ &$\cdot$ &$\cdot$ &$\cdot$ &$\cdot$ &$\cdot$ &	1&$\cdot$ &$\cdot$ &$\cdot$ &$\cdot$ &	1&$\cdot$ &$\cdot$ &$\cdot$ &	1&$\cdot$ &$\cdot$ &$\cdot$ &$\cdot$ &$\cdot$ &$\cdot$ &$\cdot$ &$\cdot$ &$\cdot$ &$\cdot$ &$\cdot$ &$\cdot$	 &$\cdot$ &$\cdot$ &	1&	1&$\cdot$ &$\cdot$ &	1&	1&$\cdot$ &$\cdot$ &$\cdot$ &	1&$\cdot$ &$\cdot$ &$\cdot$ \\
1&$\cdot$ &$\cdot$ &$\cdot$ &$\cdot$ &$\cdot$ &$\cdot$ &	1&$\cdot$ &$\cdot$ &$\cdot$ &$\cdot$ &$\cdot$ &$\cdot$ &	1&$\cdot$ &$\cdot$ &$\cdot$ &	1&$\cdot$ &$\cdot$ &$\cdot$ &$\cdot$ &$\cdot$ &$\cdot$ &$\cdot$ &	1&$\cdot$ &$\cdot$ &	1&$\cdot$ &$\cdot$ &$\cdot$ &$\cdot$ &$\cdot$ &	1&$\cdot$ &$\cdot$ &	1&$\cdot$ &$\cdot$ &$\cdot$ &	1&$\cdot$ &$\cdot$ \\ 
$\cdot$ &$\cdot$ &	1&	1&$\cdot$ &$\cdot$ &$\cdot$ &$\cdot$ &$\cdot$ &$\cdot$ &	1&	1&$\cdot$ &$\cdot$ &$\cdot$ &$\cdot$ &$\cdot$ &$\cdot$ &$\cdot$ &	1&$\cdot$ &$\cdot$ &$\cdot$ &$\cdot$ &$\cdot$ &$\cdot$ &$\cdot$ &$\cdot$ &$\cdot$ &$\cdot$ &$\cdot$ &$\cdot$ &$\cdot$ &$\cdot$ &$\cdot$ &$\cdot$ &$\cdot$ &$\cdot$ &$\cdot$ &	1&$\cdot$ &	1&	1&$\cdot$ &	1\\ 
$\cdot$ &$\cdot$ &$\cdot$ &$\cdot$ &	1&$\cdot$ &$\cdot$ &	1&$\cdot$ &$\cdot$ &$\cdot$ &	1&$\cdot$ &$\cdot$ &$\cdot$ &$\cdot$ &	1&$\cdot$ &$\cdot$ &$\cdot$ &	1&$\cdot$ &$\cdot$ &$\cdot$ &$\cdot$ &	1&$\cdot$ &$\cdot$ &	1&$\cdot$ &$\cdot$ &$\cdot$ &	1&$\cdot$ &$\cdot$ &$\cdot$ &$\cdot$ &$\cdot$ &$\cdot$ &$\cdot$ &$\cdot$ &$\cdot$ &$\cdot$ &	1&$\cdot$ \\
$\cdot$ &$\cdot$ &$\cdot$ &$\cdot$ &	1&$\cdot$ &	 &$\cdot$ &	1&$\cdot$ &	1&$\cdot$ &$\cdot$ &$\cdot$ &$\cdot$ &	1&$\cdot$ &$\cdot$ &$\cdot$ &$\cdot$ &$\cdot$ &	1&$\cdot$ &$\cdot$ &$\cdot$ &$\cdot$ &	1&$\cdot$ &$\cdot$ &$\cdot$ &	1&$\cdot$ &$\cdot$ &$\cdot$ &	1&$\cdot$ &$\cdot$ &$\cdot$ &$\cdot$ &$\cdot$ &	1&$\cdot$ &$\cdot$ &$\cdot$ &$\cdot$ \\
$\cdot$ &$\cdot$ &	1&$\cdot$ &$\cdot$ &$\cdot$ &	1&$\cdot$ &$\cdot$ &$\cdot$ &$\cdot$ &$\cdot$ &	1&	1&$\cdot$ &$\cdot$ &$\cdot$ &$\cdot$ &$\cdot$ &$\cdot$ &$\cdot$ &$\cdot$ &	1&$\cdot$ &$\cdot$ &$\cdot$ &$\cdot$ &$\cdot$ &$\cdot$ &	1&	1&	1&$\cdot$ &$\cdot$ &$\cdot$ &$\cdot$ &$\cdot$ &$\cdot$ &$\cdot$ &$\cdot$ &$\cdot$ &$\cdot$ &$\cdot$ &	1&$\cdot$ \\
$\cdot$ &$\cdot$ &$\cdot$ &	1&$\cdot$ &	1&$\cdot$ &$\cdot$ &$\cdot$ &$\cdot$ &$\cdot$ &$\cdot$ &	1&$\cdot$ &	1&$\cdot$ &$\cdot$ &$\cdot$ &$\cdot$ &$\cdot$ &$\cdot$ &$\cdot$ &$\cdot$ &	1&$\cdot$ &	1&$\cdot$ &	1&$\cdot$ &$\cdot$ &$\cdot$ &$\cdot$ &$\cdot$ &	1&$\cdot$ &$\cdot$ &$\cdot$ &$\cdot$ &$\cdot$ &$\cdot$ &	1&$\cdot$ &$\cdot$ &$\cdot$ &$\cdot$ \\ \hline
$\cdot$ &	1&$\cdot$ &$\cdot$ &	1&$\cdot$ &$\cdot$ &$\cdot$ &$\cdot$ &	1&$\cdot$ &$\cdot$ &	1&$\cdot$ &$\cdot$ &$\cdot$ &$\cdot$ &$\cdot$ &$\cdot$ &$\cdot$ &$\cdot$ &$\cdot$ &$\cdot$ &$\cdot$ &	1&$\cdot$ &$\cdot$ &$\cdot$ &$\cdot$ &$\cdot$ &$\cdot$ &$\cdot$ &$\cdot$ &$\cdot$ &$\cdot$ &$\cdot$ &$\cdot$ &$\cdot$ &$\cdot$ &$\cdot$ &	1&	1&	1&	1&$\cdot$ \\ 
1&$\cdot$ &$\cdot$ &$\cdot$ &$\cdot$ &$\cdot$ &	1&$\cdot$ &$\cdot$ &$\cdot$ &$\cdot$ &$\cdot$ &$\cdot$ &$\cdot$ &$\cdot$ &	1&$\cdot$ &$\cdot$ &$\cdot$ &$\cdot$ &	1&$\cdot$ &$\cdot$ &	1&$\cdot$ &	1&$\cdot$ &$\cdot$ &$\cdot$ &$\cdot$ &$\cdot$ &	1&$\cdot$ &$\cdot$ &	1&$\cdot$ &$\cdot$ &$\cdot$ &$\cdot$ &$\cdot$ &$\cdot$ &$\cdot$ &	1&$\cdot$ &$\cdot$ \\
$\cdot$ &$\cdot$ &$\cdot$ &	1&$\cdot$ &$\cdot$ &	1&$\cdot$ &$\cdot$ &$\cdot$ &$\cdot$ &$\cdot$ &	1&$\cdot$ &$\cdot$ &$\cdot$ &	1&$\cdot$ &	1&$\cdot$ &$\cdot$ &	1&$\cdot$ &$\cdot$ &$\cdot$ &$\cdot$ &	1&$\cdot$ &$\cdot$ &$\cdot$ &$\cdot$ &$\cdot$ &$\cdot$ &$\cdot$ &$\cdot$ &$\cdot$ &$\cdot$ &	1&$\cdot$ &$\cdot$ &$\cdot$ &$\cdot$ &$\cdot$ &$\cdot$ &	1\\
$\cdot$ &	1&$\cdot$ &$\cdot$ &$\cdot$ &$\cdot$ &$\cdot$ &	1&$\cdot$ &$\cdot$ &$\cdot$ &	1&$\cdot$ &	1&$\cdot$ &$\cdot$ &$\cdot$ &$\cdot$ &$\cdot$ &$\cdot$ &$\cdot$ &$\cdot$ &$\cdot$ &	1&$\cdot$ &$\cdot$ &$\cdot$ &	1&$\cdot$ &$\cdot$ &	1&$\cdot$ &$\cdot$ &$\cdot$ &$\cdot$ &	1&$\cdot$ &$\cdot$ &$\cdot$ &$\cdot$ &$\cdot$ &$\cdot$ &$\cdot$ &$\cdot$ &	1\\
$\cdot$ &$\cdot$ &$\cdot$ &	1&$\cdot$ &$\cdot$ &$\cdot$ &$\cdot$ &	1&	1&$\cdot$ &$\cdot$ &$\cdot$ &$\cdot$ &	1&$\cdot$ &$\cdot$ &$\cdot$ &$\cdot$ &$\cdot$ &	1&$\cdot$ &$\cdot$ &$\cdot$ &$\cdot$ &$\cdot$ &$\cdot$ &$\cdot$ &	1&$\cdot$ &	1&$\cdot$ &$\cdot$ &$\cdot$ &$\cdot$ &$\cdot$ &$\cdot$ &$\cdot$ &	1&	1&$\cdot$ &$\cdot$ &$\cdot$ &$\cdot$ &$\cdot$ \\
$\cdot$ &$\cdot$ &$\cdot$ &$\cdot$ &	1&	1&$\cdot$ &$\cdot$ &$\cdot$ &$\cdot$ &$\cdot$ &	1&$\cdot$ &$\cdot$ &$\cdot$ &	1&$\cdot$ &$\cdot$ &	1&$\cdot$ &$\cdot$ &$\cdot$ &	1&$\cdot$ &$\cdot$ &$\cdot$ &$\cdot$ &$\cdot$ &$\cdot$ &	1&$\cdot$ &$\cdot$ &$\cdot$ &$\cdot$ &$\cdot$ &$\cdot$ &	1&$\cdot$ &$\cdot$ &	1&$\cdot$ &$\cdot$ &$\cdot$ &$\cdot$ &$\cdot$ \\ \hline
1&$\cdot$ &$\cdot$ &$\cdot$ &$\cdot$ &	1&$\cdot$ &$\cdot$ &$\cdot$ &$\cdot$ &$\cdot$ &$\cdot$ &$\cdot$ &$\cdot$ &$\cdot$ &$\cdot$ &	1&$\cdot$ &$\cdot$ &$\cdot$ &$\cdot$ &	1&	1&$\cdot$ &$\cdot$ &$\cdot$ &$\cdot$ &	1&	1&$\cdot$ &	1&$\cdot$  &$\cdot$ &$\cdot$  & $\cdot$ & $\cdot$ &$\cdot$  &$\cdot$  &$\cdot$ &$\cdot$ &$\cdot$ &1&$\cdot$ &$\cdot$ &$\cdot$ \\
$\cdot$ &	1&$\cdot$ &$\cdot$ &$\cdot$ &$\cdot$ &$\cdot$ &$\cdot$ &	1&$\cdot$ &	1&$\cdot$ &$\cdot$ &$\cdot$ &	1&$\cdot$ &$\cdot$ &$\cdot$ &$\cdot$ &$\cdot$ &$\cdot$ &$\cdot$ &	1&$\cdot$ &$\cdot$ &	1&$\cdot$ &$\cdot$ &$\cdot$ &$\cdot$ &$\cdot$  &	1&$\cdot$  & $\cdot$ & $\cdot$ &$\cdot$ &	1&$\cdot$ & $\cdot$ &$\cdot$ &$\cdot$ &$\cdot$ &$\cdot$ &$\cdot$ &	1\\
$\cdot$&$\cdot$ &	1&$\cdot$ &$\cdot$ &	1&$\cdot$ &$\cdot$ &$\cdot$ &$\cdot$ &$\cdot$ &$\cdot$ &	1&$\cdot$ &$\cdot$ &	1&$\cdot$ &	1&$\cdot$ &$\cdot$ &	1&$\cdot$ &$\cdot$ &$\cdot$ &$\cdot$ &$\cdot$ &$\cdot$ &$\cdot$ &$\cdot$ &$\cdot$ &$\cdot$ &$\cdot$  &	1&$\cdot$  & $\cdot$ &$\cdot$ &$\cdot$ &$\cdot$  &	1&$\cdot$ &$\cdot$ &$\cdot$ &$\cdot$ &$\cdot$ &	1\\
$\cdot$&$\cdot$ &$\cdot$ &$\cdot$ &	1&$\cdot$ &	1&$\cdot$ &$\cdot$ &$\cdot$ &	1&$\cdot$ &$\cdot$ &$\cdot$ &$\cdot$ &$\cdot$ &	1&	1&$\cdot$ &$\cdot$ &$\cdot$ &$\cdot$ &$\cdot$ &	1&$\cdot$ &$\cdot$ &$\cdot$ &$\cdot$ &$\cdot$ &$\cdot$ &$\cdot$  & $\cdot$ & $\cdot$ &	1&$\cdot$ &	1&$\cdot$  & $\cdot$ &$\cdot$  &	1&$\cdot$ &$\cdot$ &$\cdot$ &$\cdot$ &$\cdot$ \\
$\cdot$&$\cdot$ &	1&$\cdot$ &$\cdot$ &$\cdot$ &$\cdot$ &	1&$\cdot$ &	1&$\cdot$ &$\cdot$ &$\cdot$ &	1&$\cdot$ &$\cdot$ &$\cdot$ &$\cdot$ &$\cdot$ &$\cdot$ &$\cdot$ &	1&$\cdot$ &$\cdot$ &$\cdot$ &	1&$\cdot$ &$\cdot$ &$\cdot$ &$\cdot$ & $\cdot$ & $\cdot$ &$\cdot$  &$\cdot$ &	1& $\cdot$ &$\cdot$  &	1&$\cdot$ &	1&$\cdot$ &$\cdot$ &$\cdot$ &$\cdot$ &$\cdot$ \\ \hline
 $\cdot$&$\cdot$ &	1&$\cdot$ &$\cdot$ &$\cdot$ &$\cdot$ &$\cdot$ &	1&	1&$\cdot$ &$\cdot$ &$\cdot$ &$\cdot$ &$\cdot$ &	1&$\cdot$ &$\cdot$ &	1&$\cdot$ &$\cdot$ &$\cdot$ &$\cdot$ &$\cdot$ &$\cdot$ &$\cdot$ &$\cdot$ &	1&$\cdot$ &$\cdot$ &$\cdot$ &$\cdot$  & $\cdot$ &	1&$\cdot$  &	1&$\cdot$ &$\cdot$ &$\cdot$  &$\cdot$ &$\cdot$ &$\cdot$ &$\cdot$ &	1&$\cdot$ \\
$\cdot$ &$\cdot$ &$\cdot$ &	1&$\cdot$ &$\cdot$ &$\cdot$ &	1&$\cdot$ &	1&$\cdot$ &$\cdot$ &$\cdot$ &$\cdot$ &$\cdot$ &$\cdot$ &1&1&$\cdot$ &$\cdot$ &$\cdot$ &$\cdot$ &$\cdot$ &$\cdot$ &$\cdot$ &$\cdot$ &$\cdot$ &$\cdot$ &$\cdot$ &	1& $\cdot$ &	1&$\cdot$ &$\cdot$  &$\cdot$  & $\cdot$ &	1& $\cdot$ &$\cdot$  &$\cdot$ &	1&$\cdot$ &$\cdot$ &$\cdot$ &$\cdot$ \\
$\cdot$ &	1&$\cdot$ &$\cdot$ &$\cdot$ &	1&$\cdot$ &$\cdot$ &$\cdot$ &$\cdot$ &$\cdot$ &	1&$\cdot$ &$\cdot$ &	1&$\cdot$ &$\cdot$ &	1&$\cdot$ &$\cdot$ &$\cdot$ &$\cdot$ &$\cdot$ &$\cdot$ &$\cdot$ &$\cdot$ &	1&$\cdot$ &$\cdot$ &$\cdot$ & $\cdot$ &$\cdot$ &$\cdot$  &$\cdot$ &	1&$\cdot$  & $\cdot$ &	1&$\cdot$  &$\cdot$ &$\cdot$ &$\cdot$ &$\cdot$ &	1&$\cdot$ \\
$\cdot$&	1&$\cdot$ &$\cdot$ &$\cdot$ &$\cdot$ &	1&$\cdot$ &$\cdot$ &$\cdot$ &	1&$\cdot$ &$\cdot$ &	1&$\cdot$ &$\cdot$ &$\cdot$ &$\cdot$ &	1&$\cdot$ &$\cdot$ &$\cdot$ &$\cdot$ &$\cdot$ &$\cdot$ &$\cdot$ &$\cdot$ &$\cdot$ &	1&$\cdot$ & $\cdot$ &$\cdot$  &	1&$\cdot$  &$\cdot$  &$\cdot$  &$\cdot$  &$\cdot$  &	1&$\cdot$ &	1&$\cdot$ &$\cdot$ &$\cdot$ &$\cdot$ \\ \hline
1&	1&$\cdot$ &$\cdot$ &$\cdot$ &$\cdot$ &$\cdot$ &$\cdot$ &$\cdot$ &$\cdot$ &$\cdot$ &$\cdot$ &	1&$\cdot$ &$\cdot$ &$\cdot$ &$\cdot$ &$\cdot$ &$\cdot$ &	1&$\cdot$ &$\cdot$ &$\cdot$ &$\cdot$ &$\cdot$ &$\cdot$ &$\cdot$ &$\cdot$ &	1&	1&$\cdot$ &$\cdot$ &$\cdot$ &	1&	1&$\cdot$ &$\cdot$ &$\cdot$ &$\cdot$	 &1&$\cdot$ &$\cdot$ &$\cdot$ &$\cdot$ &$\cdot$	 \\
1&$\cdot$ &	1&$\cdot$ &$\cdot$ &$\cdot$ &$\cdot$ &$\cdot$ &$\cdot$ &$\cdot$ &$\cdot$ &	1&$\cdot$ &$\cdot$ &$\cdot$ &$\cdot$ &$\cdot$ &$\cdot$ &$\cdot$ &$\cdot$ &$\cdot$ &	1&$\cdot$ &	1&	1&$\cdot$ &$\cdot$ &$\cdot$ &$\cdot$ &$\cdot$ &$\cdot$ &$\cdot$ &$\cdot$ &$\cdot$ &$\cdot$ &$\cdot$ &	1&$\cdot$ &	1&$\cdot$ &	1&$\cdot$ &$\cdot$ &$\cdot$ &$\cdot$	 \\
$\cdot$ &$\cdot$ &$\cdot$ &$\cdot$ &$\cdot$ &$\cdot$ &	1&	1&$\cdot$ &$\cdot$ &$\cdot$ &$\cdot$ &$\cdot$ &$\cdot$ &	1&	1&$\cdot$ &	1&$\cdot$ &	1&$\cdot$ &$\cdot$ &$\cdot$ &$\cdot$ &	1&$\cdot$ &$\cdot$ &$\cdot$ &$\cdot$ &$\cdot$ &	1&$\cdot$ &$\cdot$ &$\cdot$ &$\cdot$ &$\cdot$ &$\cdot$ &$\cdot$ &$\cdot$ &$\cdot$ &$\cdot$ &	1&$\cdot$ &$\cdot$ &$\cdot$	 \\ \hline
$\cdot$ &$\cdot$ &$\cdot$ &$\cdot$ &$\cdot$ &	1&$\cdot$ &$\cdot$ &	1&$\cdot$ &$\cdot$ &$\cdot$ &$\cdot$ &	1&$\cdot$ &$\cdot$ &	1&$\cdot$ &	1&	1&$\cdot$ &$\cdot$ &$\cdot$ &$\cdot$ &	1&	1&$\cdot$ &$\cdot$ &$\cdot$ &$\cdot$ &$\cdot$ &$\cdot$ &$\cdot$ &$\cdot$ &$\cdot$ &$\cdot$ &$\cdot$ &$\cdot$ &$\cdot$	 &$\cdot$ &$\cdot$ &$\cdot$	 &	1&$\cdot$	 &$\cdot$	 \\
1&	 &$\cdot$ &	1&$\cdot$ &$\cdot$ &$\cdot$ &$\cdot$ &$\cdot$ &$\cdot$ &	1&$\cdot$ &$\cdot$ &$\cdot$ &$\cdot$ &$\cdot$ &$\cdot$ &$\cdot$ &$\cdot$ &$\cdot$ &	1&$\cdot$ &	1&$\cdot$ &	1&$\cdot$ &$\cdot$ &$\cdot$ &$\cdot$ &$\cdot$ &$\cdot$ &$\cdot$ &$\cdot$ &$\cdot$ &$\cdot$ &	1&$\cdot$ &	1&$\cdot$ &$\cdot$ &$\cdot$ &$\cdot$ &$\cdot$ &	1&$\cdot$	 \\ \hline
1&$\cdot$ &$\cdot$ &$\cdot$ &	1&$\cdot$ &$\cdot$ &$\cdot$ &$\cdot$ &	1&$\cdot$ &$\cdot$ &$\cdot$ &$\cdot$ &$\cdot$ &$\cdot$ &$\cdot$ &$\cdot$ &$\cdot$ &	1&$\cdot$ &$\cdot$ &$\cdot$ &$\cdot$ &$\cdot$ &$\cdot$ &	1&	1&$\cdot$ &$\cdot$ &$\cdot$ &	1&	1&$\cdot$ &$\cdot$ &$\cdot$ &$\cdot$ &$\cdot$ &$\cdot$ &$\cdot$ &$\cdot$ &$\cdot$ &$\cdot$ &$\cdot$ &	1\\ \hline
\end{tabular}
\end{tiny}
\caption{Constructed incidence matrix of the biplane ${\cal B}_{9e}$, with canonical part omitted.}\label{Fig1}
\end{center}
\end{figure}

\section{Conjecture on non-transversal vectors}

Further inspection of canonical incidence matrices of biplanes of order $4$ has shown a remarkable property. In every subset ${\cal M}_i^v$ there is a vector $\alpha \in  {\cal M}_i^v$, $i=7,\ldots,16$ with a characteristic to have scalar product equal to $2$ with exactly $5$ vectors in any other subset. Exactly these $10$ vectors form sole canonical incidence matrix of biplanes of this order representing the biplane ${\cal B}_{4c}$. Based on these facts, we introduce the notion of {\it non-transversal} vectors of a biplane.
\begin{definition}
Let ${\cal B}$ be a biplane of order $k-2$. The set of vectors $\{a_1,a_2,\ldots,a_m\} \subset {\cal M}^v$ such that 
$$
f(a_i, a_j) =2, \enspace \forall i,j, \enspace i\neq j
$$ 
we shall call the set of non-transversal vectors of a biplane ${\cal B}$. 
\end{definition}
Clearly, once having $\binom{k-1}{2}$ non-transversal vectors, a biplane of order $k-2$ is constructed. 

The $j$-th {\it level set} $f_{a,j}^{-1}(2)$ of a vector $a \in {\cal M}_i^{v}$ is the set
$$
f_{a,j}^{-1}(2):=\{b:f(a,b)=2, \enspace b\in {\cal M}_j^{v}\},
$$
where $i.j \in \{k+1,\ldots,v\}$ and $i \neq j$.
Concretely, in case of biplanes of order 4 for the only vector $\alpha$ in a particular subset ${\cal \grave{M}}_i^{16}$ it holds 
$$
|f_{a,j}^{-1}(2)|=1, \enspace j=7,\ldots,16, \enspace i \neq j.
$$
This fact is in line with the following more general Conjecture \ref{Conj1}. However, for biplanes of higher orders there are subsets without such regular vectors. We shall call them {\it exceptional subsets}. More precisely, our experiments show that these are subsets ${\cal \grave{M}}_i^{v} \subset {\cal \grave{M}}^{v}, \enspace i \in E$, $$E=\{2k-2,3k-5,4k-9,\ldots,v-1,v\}.$$ The second difference is that the regularity does not hold for every $j$-th level set but, impressively, there are still vectors having many $j$-th level sets with the same cardinality - which makes us able to recognize and isolate them.

Let $a$ be vector in a biplane's $2$-space, $r,q(a)  \in \mathbb N$ and let the set $Q(a)$ be defined as $$Q(a):=\{(1\mod v)+k,(2\mod v)+k,\ldots,(i\mod v)+k,\ldots,(r\mod v)+k\},$$
with $i \notin E$. Then, our interest is in vectors $a \in {\cal \grave{M}}_i^v, \enspace \forall i \notin E$ having the property 
\begin{eqnarray} \label{RegVect}
|f_{a,j}^{-1}(2)|=q(a), \enspace j \in Q(a).
\end{eqnarray}
Now we define the sets ${\cal \grave{M}}^{v(\hat{a})}, {\cal \grave{M}}^{v(a)}$ as follows.
\begin{eqnarray*}
{\cal \grave{M}}^{v(\hat{a})}:&=&\{a \in {\cal \grave{M}}_i^v :    |f_{a,j}^{-1}(2)|=q(a), \enspace  \forall i \notin E, \enspace q(a) \in \mathbb N, \enspace j \in Q(a) \}\\
{\cal \grave{M}}^{v({a})}:&=&{\cal \grave{M}}^{v(\hat{a})} \cup {\cal \grave{M}}_i^v, \enspace i \in E
\end{eqnarray*}

\begin{conjecture}\label{Conj1}
Let ${\cal B}$ be a biplane and $a_1,a_2,\ldots,a_m \in {\cal \grave{M}}^v$ vectors obeying the property (\ref{RegVect}). Let ${\cal \grave{M}}^{v({a_1}{a_2}\ldots{a_m})}$ be the set defined as
$$
{\cal \grave{M}}^{v({a_1}{a_2}\ldots{a_m})}:={\cal \grave{M}}^{v(\hat{a_1})} \cup \ldots {\cal \grave{M}}^{v(\hat{a_m})} \cup {\cal \grave{M}}_i^v, \enspace i \in E.
$$
Then there exists the natural number $r>0$ such that $2$-spaces ${\cal \grave{M}}^{v({a_1}{a_2}\ldots{a_m})}$ and ${\cal \grave{M}}^{v}$ contain biplanes of the same isomorphism class.
\end{conjecture}

Thus, biplanes of order $4$ have one vector $\alpha$ with constant cardinality within every $j$-th level set (equal to 5 in respect of ${\cal M}^v$ and equal to 1 in respect of ${\cal \grave{M}}^v$). Within biplanes of order $7$, $9$ and $11$ there are two types of vectors having the property (\ref{RegVect}). We let $\alpha$ and $\beta$ denote such vectors. Table \ref{table2} shows the number of these vectors in a particular subset ${\cal \grave{M}}_i^{v(a)}$ as well as the cardinal numbers $q(\alpha)$, $q(\beta)$ of their $j$-th level space, where $i=k+1,\ldots, v, \enspace i \notin E$, $j=k+1, \ldots, v$.

Obviously, if Conjecture \ref{Conj1} holds true then it substantially reduces constructions of isomorphic structures. We successfully use this conjecture for very efficient constructions of all biplanes of order 7 as well as biplanes ${\cal B}_{9c}$ and ${\cal B}_{9e}$. If Conjecture \ref{Conj1} is true then the only biplane of order 9 admitting both symmetry of incidence matrix and trace of the incidence matrix equal 56, is ${\cal B}_{9e}$ with the automorphism group order $| Aut({\cal B}_{9e})|=80640$.

\begin{table} [h]
\begin{center}
\begin{tabular}{rrrrr}
$k-2$ &$q(\alpha)$& $q(\beta)$& $|{\cal \grave{M}}_i^{(\alpha)}|$ & $|{\cal \grave{M}}_i^{(\beta)}|$\\
7 & 30&24&10& 60\\
9 & 1116&1098&315 & 2520\\
11 &87235&87178& 18144 &  181440\\
\end{tabular}
\end{center}
\caption{Cardinalities of $j$-th level space of $\alpha$ and $\beta$ types of vectors within biplanes.} \label{table2}
\end{table}

\section{Concluding remarks}

Two notable properties among biplane's incidence matrices are their symmetry and the trace equal to the number of points. Although many regularities of biplane's incidence matrix is getting lost within large parameters, trace equal to the number of points seems persistent. Both symmetry and trace equal to the number of points are present within large biplanes but the latter without known exception. 

In general these are independent properties. For example, no canonical incidence matrix representing biplane of order 7 is symmetric but some of duals admit trace equal to the number of points. On the contrary, all constructed incidence matrices of ${\cal B}_{9e}$ are symmetric and all have trace equal to $v$. Also, all constructed incidence matrices of  ${\cal B}_{9c}$ have trace equal to $v$ and none of these matrices is symmetric. Thus, in all our constructions incidence matrix symmetry behaves like a structural invariant of biplanes while it is not the case with trace. Possibly, statements and ideas demonstrated in this work can be extended to other finite geometries, primarily projective planes and triplanes.

\section*{Acknowledgment}

The author is thankful to Professor Dragutin Svrtan from the University of Zagreb for advices regarding this subject.

\end{document}